\documentclass[12pt]{amsart}
\usepackage{amsthm,amsmath,amssymb,cite}
\usepackage{graphicx}
\usepackage{epsfig}
\usepackage{amsmath,amscd}

\theoremstyle{plain}
\newtheorem{theorem}{Theorem}[section]
\newtheorem{lemma}[theorem]{Lemma}
\newtheorem{remark}[theorem]{Remark}

\newtheorem{definition}{Definition}[section]

\begin{document}
\title{Some completely monotonic properties for the  $(p,q )$-gamma function}
\author{Valmir Krasniqi}
\address{Department of Mathematics and Computer Sciences, University of Prishtina, Republic of Kosova}
\email{vali.99@hotmail.com}
\author{Faton Merovci}
\address{Department of Mathematics, University of Prishtina, Republic of Kosova}
\email{fmerovci@yahoo.com}

\begin{abstract} 
 It is defined $\Gamma_{p,q}$ function, a generalize of $\Gamma$ function. Also, we defined $\psi_{p,q}$-analogue of the psi function as the log derivative of $\Gamma_{p,q}$. For the $\Gamma_{p,q}$ -function, are given some properties related to convexity, log-convexity and completely monotonic function. Also, some properties of $\psi_{p,q} $ analog of the $\psi$ function have been established. As an application, when $p\to \infty,  q\to 1,$ we obtain all result of \cite{Valmir1} and  \cite{SHA}.
\end{abstract}

%%\date{\today}
\subjclass[2010]{
{\sl MSC 2010}: 33B15; 26A51; 26A48}
\keywords{Completely monotonic function, logarithmically completely monotonic function, $(p,q) -$Gamma function, $(p,q) -$psi function, generalization  inequality}
\maketitle

\section{Introduction}
The Euler gamma function $\Gamma(x)$ is defined for $x>0$ by
$$\Gamma(x)=\int_0^\infty t^{x-1}e^{-t}dt.$$
The digamma (or psi) function is defined for positive real numbers
$x$ as the logarithmic derivative of Euler's gamma function, that is
$\displaystyle \psi(x)=\frac{d}{dx}\ln \Gamma(x)=\frac{\Gamma'(x)}{\Gamma(x)}$. The following integral and series
representations are valid (see \cite{AS}):
\begin{equation}\label{eqpsi1}
\psi(x)=-\gamma+\int_0^\infty\frac{e^{-t}-e^{-xt}}{1-e^{-t}}dt=-\gamma-\frac{1}{x}+\sum_{n\geq1}\frac{x}{n(n+x)},
\end{equation}
where $\gamma=0.57721\cdots$ denotes Euler's constant.

Euler, gave another equivalent definition for the $\Gamma(x)$ (see \cite{AP},\cite{SA})
\begin{equation}\label{eqgammap1}
\Gamma_p(x)=\frac{p!p^x}{x(x+1)\cdot \ldots \cdot (x+p)}=\frac{p^x}{x(1+\frac{x}{1})\cdot \ldots \cdot(1+\frac{x}{p})},\ x>0,
\end{equation}
where $p$ is positive integer, and
\begin{equation}\label{eqgammap2}
\Gamma(x) =\lim_{p\rightarrow\infty}\Gamma_p(x).
\end{equation}

The $p$-analogue of the psi function as the
logarithmic derivative of the $\Gamma_p$ function (see \cite{Valmir1}), that is
\begin{equation}\label{psi_p1}
\psi_p(x)=\frac{d}{dx} \ln \Gamma_p(x)=\frac{\Gamma'_p(x)}{\Gamma_p(x)}.
\end{equation}

The following representations are valid:
\begin{equation}\label{eq6}
\Gamma_p(x)=\int_{0}^{p}\Big(1-\frac{t}{p}\Big)^pt^{x-1}dt,
\end{equation}
\begin{equation}\label{eq6-1}
\psi_p(x)=\ln p-\int_{0}^{\infty}\frac{e^{-xt}(1-e^{-(p+1)t})}{1-e^{-t}}dt
\end{equation}
and
\begin{equation}\label{eq7}
\psi_p^{(m)}(x)=(-1)^{m+1} \int_{0}^{\infty}\frac{t^m\cdot e^{-xt}}{1-e^{-t}}(1-e^{-(p+1)t})dt.
\end{equation}
Jackson (see \cite{K,K2,KR,SKS}) defined the $q$-analogue of the
gamma function as
\begin{equation}\label{eqgamma1}
\Gamma_q(x)=\frac{(q;q)_\infty}{(q^x;q)_\infty}(1-q)^{1-x},\,
0<q<1,\end{equation}
and
\begin{equation}\label{eqgamma2}
\Gamma_q(x)=\frac{(q^{-1};q^{-1})_\infty}{(q^{-x};q^{-1})_\infty}(q-1)^{1-x}q^{\binom{x}{2}},\,q>1,
\end{equation}
where $(a;q)_\infty=\prod_{j\geq0}(1-aq^j)$.

The $q$-gamma function has the following integral representation (see \cite{As})
$$\Gamma_q(t)=\int_{0}^{\infty}x^{t-1}E_q^{-qx}d_qx,$$
where $E_q^x=\sum_{j=0}^{\infty}q^{\frac{j(j-1)}{2}}\frac{x^j}{[j]!}=(1+(1-q)x)_q^{\infty},$
which is the $q$-analogue of the classical exponential function.
The $q$-analogue of the psi function is defined for $0<q<1$ as the
logarithmic derivative of the $q$-gamma function, that is,
$\psi_q(x)=\frac{d}{dx} \log \Gamma_q(x)$.
Many properties of the $q$-gamma function were derived by Askey
\cite{As}. It is well known that $\Gamma_q(x)\rightarrow\Gamma(x)$
and $\psi_q(x)\rightarrow\psi(x)$ as $q\rightarrow1^{-}$. From
\eqref{eqgamma1}, for $0<q<1$ and $x>0$ we get
\begin{equation}\label{eqpsiq1}
\psi_q(x)=-\log(1-q)+\log
q\sum_{n\geq0}\frac{q^{n+x}}{1-q^{n+x}}=-\log(1-q)+\log
q\sum_{n\geq1}\frac{q^{nx}}{1-q^n},
\end{equation}
and from \eqref{eqgamma2} for $q>1$ and $x>0$ we obtain
\begin{equation}\label{eqpsq2}
\begin{array}{ll}
\psi_q(x)&=-\log(q-1)+\log q\left(x-\frac{1}{2}-\sum\limits_{n\geq0}\frac{q^{-n-x}}{1-q^{-n-x}}\right)\\
&=-\log(q-1)+\log q
\left(x-\frac{1}{2}-\sum\limits_{n\geq1}\frac{q^{-nx}}{1-q^{-n}}\right).
\end{array}
\end{equation}
A Stieltjes integral representation for $\psi_q(x)$ with $0<q<1$ is
given in \cite{IM}. It is well-known that $\psi'$ is strictly
completely monotonic on $(0,\infty)$, that is,
$$(-1)^n(\psi'(x))^{(n)}>0\quad \mbox{for }x>0\mbox{ and }n\geq0,$$
see \cite[Page 260]{AS}. From \eqref{eqpsiq1} and \eqref{eqpsq2} we
conclude that $\psi_q'$ has the same property for any $q > 0$,
$$(-1)^n(\psi_q'(x))^{(n)}>0\quad \mbox{for }x>0\mbox{ and }n\geq0.$$
If $q\in (0,1)$, using the second representation of $\psi_q(x)$
given in \eqref{eqpsiq1} can be shown that
\begin{equation}\label{in07}
\psi_q^{(k)}(x)=\log^{k+1}q\sum\limits_{n\geq1}\frac{n^k\cdot
q^{nx}}{1-q^n},
\end{equation}
and hence $(-1)^{k-1}\psi_q^{(k)}(x)>0$ with $x>1$, for all
$k\geq1$. If $q>1$, from the second representation of $\psi_q(x)$
given in \eqref{eqpsq2} we obtain
\begin{equation}\label{eqpsq03}
\psi'_q(x)=\log q\Big(1+\sum\limits_{n\geq
1}\frac{nq^{-nx}}{1-q^{-nx}}\Big),
\end{equation}
and for $k\geq 2$,
\begin{equation}\label{eqpsq3}
\psi^{(k)}_q(x)=(-1)^{k-1} \log^{k+1} q\sum\limits_{n\geq 1}\frac{
n^kq^{-nx}}{1-q^{-nx}}.
\end{equation}
Hence, $(-1)^{k-1}\psi^{(k)}_q(x)>0$ with $x>0$, for all $q>1$.

\begin{definition} For $x>0$, $p\in N$  and for $q \in (0,1)$
\begin{equation}\label{ekuacioni} \Gamma_{p,q}(x)=\frac{[p]_q^{x}[p]_q!}{[x]_q[x+1]_q\cdots [x+p]_q},
\end{equation}where $[p]_q=\frac{1-q^p}{1-q}$.
\end{definition}

It is easy to see than $\Gamma_{p,q}(x)$ fit into the following commutative diagrams:
$$\begin{CD}
\Gamma_{p,q}(x) @>p\to\infty >> \Gamma_q(x)\\
@VVq\to 1V @VVq\to 1V\\
\Gamma_p(x) @>p\to \infty>> \Gamma(x)
\end{CD}$$

We define  $(p,q)$-analogue of the psi function   as the
logarithmic derivative of the $p,q$-gamma function, that is,
\begin{equation}\label{psiPQ}\psi_{p,q}(x)=\frac{d}{dx} \log \Gamma_{p,q}(x).\end{equation}

\begin{definition} The function $f$ is called log-convex if for all $\alpha, \beta >0$  such that $\alpha+\beta=1$ and for all $x,y>0$ the following inequality holds
$$\log f(\alpha x+\beta y)\leq \alpha \log f(x) + \beta \log f(y),$$
or equivalently
$$f(\alpha x+\beta y)\leq (f(x))^{\alpha}\cdot (f(y))^{\beta}.$$
\end{definition}
Now, we will give some definitions about completely monotonic function:

A function $f$ is said to be {\em completely monotonic} on an open interval $I$, if $f$ has derivatives of all orders on $I$
and satisfies
\begin{equation}\label{eq0}
(-1)^nf^{(n)}(x)\geq 0,   (x\in I, n=0,1,2,\ldots ).
\end{equation}
If the inequality \eqref{eq0} is strict, then $f$ is said to be {\em strictly completely monotonic} on $I$.

A positive function $f$ is said to be {\em logarithmically completely monotonic} (see \cite{QI}) on an open interval $I$, if $f$ satisfies
\begin{equation}\label{eq01}
(-1)^n[\ln f(x)]^{(n)}\geq 0, (x\in I, n=1,2,\ldots ).
\end{equation}
If the inequality \eqref{eq01} is strict, then $f$ is said to be {\em strictly logarithmically completely monotonic}.

Let C and L denote the set of completely monotonic functions and the set of logarithmically completely monotonic functions, respectively. The relationship between completely monotonic functions and logarithmically completely monotonic functions can be presented (see [2]) by $L\subset C$.

The following theorem gives an integral characterization of completely monotone functions.

\begin{theorem}
(Hausdorff- Bernstein- Widder Theorem) A function $ \varphi :[0,\infty)\rightarrow R$ is completely monotone on $[0,\infty)$ if and only if it is the Laplace transform of a finite non-negative Borel meassure $\mu$ on $[0,\infty),$ i.e., $ \varphi$ is of the form
\begin{equation}
\varphi (r)=\int\limits_{0}^{\infty}e^{-rt}d\mu (t).
\end{equation}
\end{theorem}

\begin{itemize}
\item A non-negative finite linear combination of completely monotone functions is completely monotone.
\item The product of two completely monotone functions is completely monotone.
\end{itemize}

\section{Main results}
\begin{lemma}\label{lema1}
For $\alpha, \beta\geq 0$  such that $\alpha+\beta=1$ we have:
\begin{equation}[1+x]_q^\alpha [1+y]_q^\beta\leq [1+\alpha x+\beta y]_q.
\end{equation}
\end{lemma}
\begin{proof}
From Youngs inequlity
\begin{equation}
[x]_q^\alpha [y]_q^\beta\leq \alpha [x]_q+\beta[y]_q,
\end{equation}
we have:
\begin{align*}[1+x]_q^\alpha [1+y]_q^\beta& \leq \alpha [1+x]_q+\beta[1+y]_q\noindent \\
\noindent &=\alpha \Big( \frac{1-q^{1+x}}{1-q}\Big)+\beta\Big( \frac{1-q^{1+y}}{1-q}\Big)\noindent \\
\noindent &=\frac{1}{1-q}\Big[1-(\alpha q^{1+x}+\beta q^{1+y})\Big]\end{align*}
We have to  prove:
\begin{equation}\alpha q^{1+x}+\beta q^{1+y}\geq q^{1+\alpha x+\beta y}. \end{equation}
From Youngs inequality we have:
$$q^{1+\alpha x+\beta y}=q((q^x)^\alpha (q^x)^\beta )\leq q(\alpha q^x+\beta q^y)=\alpha q^{1+x}+\beta q^{1+y}. $$
\end{proof}

\begin{theorem}\label{logcon0} The function
\[
\Gamma_{p,q}(x)=\frac{[p]_q^{x}[p]_q!}{[x]_q[x+1]_q\cdots [x+p]_q}
\]
is log-convex.
\end{theorem}
\begin{proof} We have to prove that for all $\alpha, \beta >0, \alpha+\beta =1, x,y>0$
\begin{equation}\label{logcon1}
\log \Gamma_{p,q}(\alpha x+\beta y)\leq \alpha \log \Gamma_{p,q}(x)+\beta \log \Gamma_{p,q}(y),
\end{equation}
which is equivalent to
\begin{equation}\label{logcon2}
\Gamma_{p,q}(\alpha x+\beta y)\leq (\Gamma_{p,q}(x))^{\alpha}\cdot (\Gamma_{p,q}(y))^{\beta}.
\end{equation}
By lemma \ref{lema1} we obtain:
\begin{equation}\label{logcon4}
\Big[1+\frac{x}{k}\Big]_q^{\alpha}\cdot \Big[1+\frac{y}{k}\Big]_q^{\beta}\leq \alpha \Big[1+\frac{x}{k}\Big]_q+\beta\Big[1+\frac{y}{k}\Big]_q=\Big[1+\frac{\alpha x+\beta y}{k}\Big]_q
\end{equation}
for all $k\geq 1, k\in \mathbf{\mathbf{N}}$.

Multiplying \eqref{logcon4} for $k=1,2,\ldots, p$ one obtains
\[
\Big[1+\frac{x}{1}\Big]_q^{\alpha} \ldots \Big[1+\frac{x}{p}\Big]_q^{\alpha}\cdot\Big[1+\frac{y}{1}\Big]_q^{\beta}\ldots \Big[1+\frac{y}{p}\Big]_q^{\beta}\leq \Big[1+\frac{\alpha x+\beta y}{1}\Big]_q \ldots \Big[1+\frac{\alpha x+\beta y}{p}\Big]_q.
\]

\noindent Now, taking the reciprocal values and multiplying by $\displaystyle [p]_q^{\alpha x+\beta y}$ one obtains \eqref{logcon2} and thus the proof is completed.
\end{proof}

\begin{lemma}
a) The function $\psi_{p,q}$ defined by \eqref{psiPQ} has the following series representation and integral representation
\begin{equation}\label{psipq}
\psi_{p,q}(x)=\ln [p]_q+\log q\sum\limits_{k=0}^{p}\frac{q^{x+k}}{1-q^{x+k}},
\end{equation}
\begin{equation}\label{psipqintegral}
\psi_{p,q}(x)=\ln [p]_q-\int\limits_{0}^{\infty}\frac{e^{-xt}}{1-e^{-t}}(1-e^{-pt})d\gamma_q(t).
\end{equation}
where $\gamma_q(t)$ is a discrete measure with positive masses $-\log q$ at the positive points $-k\log q, k=1,2,\ldots , $i.e.
\begin{equation} \gamma_q(t)=-\log q\sum\limits_{k=1}^{\infty}\delta (t+k\log q), \quad 0<q<1.\end{equation}
b) The function $\psi_{p,q}$ is increasing on $(0,\infty)$.\\
c) The function $\psi'_{p,q}$ is strictly completely monotonic on
$(0,\infty).$
\end{lemma}
\begin{proof} a) After logarithmical and derivative of \eqref{ekuacioni} we take \eqref{psipq}.

Using $\int\limits_{0}^{\infty}e^{-xt}d\gamma_q(t)=\frac{-q^x\log q}{1-q^x}, 0<q<1$ (see \cite{Integrali}) we take \eqref{psipqintegral}.

b) Let $0<x<y$. Using \eqref{psipq} we obtain
\begin{align*}
\psi_{p,q}(x)-\psi_{p,q}(y)&=\log q\sum_{k=0}^{p}\Big(\frac{1-q^{y+k}}{q^{y+k}}-\frac{1-q^{y+k}}{q^{y+k}}\Big)\\
 \notag &=\log q\sum_{k=0}^{p}\Big(\frac{q^x-q^{x+y+k}-q^y+q^{x+y+k}}{q^{x+y+k}}\Big)\\
 \notag &=\log q\sum_{k=0}^{p}\Big(\frac{q^x-q^y}{q^{x+y+k}}\Big)<0.
\end{align*}

c) Deriving $n$ times the relation \eqref{psipq}  one finds that:
\begin{equation}\label{psi_series2}
\psi_{p,q} ^{(n)}(x)=(-1)^{n+1}\int\limits_{0}^{\infty}\frac{t^ne^{-xt}}{1-e^{-t}}(1-e^{-(p+1)t})d\gamma_q(t).
\end{equation}
Hence $(-1)^n(\psi'_{p,q}(x))^{(n)}>0$, for $x>0, n\geq 0.$
\end{proof}

\begin{remark}
$\psi_{p,q}(x)$ fit into the following commutative diagrams$$\begin{CD}
\psi_{p,q}(x) @>p\to\infty >> \psi_q(x)\\
@VVq\to 1V @VVq\to 1V\\
\psi_p(x) @>p\to \infty>> \psi(x)
\end{CD}$$
\end{remark}

\section{Logarithmically completely monotonic  function}

\begin{theorem}\label{th2}
The function $G_{p,q}(x;a_1,b_1,\ldots,a_n,b_n)$ given by
\begin{equation}\label{eqGAO}
G_{p,q}(x)=G_{p,q}(x;a_1,b_1,\ldots,a_n,b_n)=\prod_{i=1}^{n}\frac{\Gamma_{p,q} (x+a_i)}{\Gamma_{p,q} (x+b_i)}, q\in (0,1)
\end{equation}
is a completely monotonic function on $(0,\infty)$, for any $a_i$ and $b_i$, $i=1,2,\ldots,n$, real numbers such that $0<a_1\leq \cdots \leq a_n$, $0<b_1\leq b_2\leq \cdots \leq b_n$ and $\sum_{i=1}^{k}a_i\leq \sum_{i=1}^{k}b_i$ for $k=1,2,\ldots,n$.
\end{theorem}
\begin{proof}
Let $h(x)=\sum_{i=1}^{n}(\log \Gamma_{p,q}(x+b_i)-\log \Gamma_{p,q}(x+a_i))$. Then for $k\geq 0$ we have
\begin{align}
\notag (-1)^k(h'(x))^{(k)}&=(-1)^k\sum_{i=1}^{n}(\psi_{p,q}^{(k)}(x+b_i)-\psi_{p,q}^{(k)}(x+a_i))\notag\\
\notag &=(-1)^k\sum_{i=1}^{n}(-1)^{k+1}\int\limits_{0}^{\infty}\frac{t^ke^{-xt}}{1-e^{-t}}(1-e^{-(p+1)t})(e^{-bi}-e^{-ai})d\gamma_q(t)\notag\\
\notag &=(-1)^{2k+1}\int\limits_{0}^{\infty}\frac{t^ke^{-xt}}{1-e^{-t}}(1-e^{-(p+1)t})\sum_{i=1}^{n}(e^{-bi}-e^{-ai})d\gamma_q(t).\notag
\end{align}
Alzer \cite{AL0} showed that if $f$ is a decreasing and convex function on $R$, then holds
\begin{equation}\label{eqal}
\sum_{i=1}^{n}f(b_i)\leq \sum_{i=1}^{n}f(a_i).
\end{equation}
Thus, since the function $z\mapsto e^{-z}, z>0$ is decreasing and convex on $R$, we have that $\sum_{i=1}^n(e^{-ai}-e^{-bi})\geq0$, so $(-1)^k(G'_{p,q}(x))^{(k)}\geq 0$ for $k\geq 0$. Hence $h'$ is completely monotonic on $(0,\infty)$. Using the fact that if $h'$ is completely monotonic function on $(0,\infty)$, then $\exp(-h)$ is also completely monotonic function on $(0,\infty)$ (see \cite{BO}), we get the desired result.
\end{proof}

\begin{theorem}
The function \begin{equation}f(x)=\frac{1}{\left(\frac{[p]_q}{[p+1]_q}\Gamma _{p,q} (x)\right)^{\frac{1}{x}}}\end{equation}
 is logarithmically completely monotonic in $(0,\infty).$
\end{theorem}
\begin{proof}
Using Leibnitz rule
$$[u(x)v(x)]^{(n)}=\sum\limits_{k=0}^{n}\binom{n}{k}u^{(k)}(x)v^{(n-k)}(x),$$
we obtain
\begin{align*}[\ln f(x)]^{(n)}&=\sum\limits_{k=0}^{n}\binom{n}{k}\Big(\frac{1}{x}\Big)^{(k)}\Big(-\ln \Gamma _{p,q}(x+1)\Big)^{(n-k)}\notag \\
\notag &=-\frac{1}{x^{n+1}}\sum\limits_{k=0}^{n}\binom{n}{k}(-1)^kk!x^{n-k}\psi _{p,q} ^{(n-k-1)}(x+1)\notag \\
\notag &= -\frac{1}{x^{n+1}}g(x)\end{align*}
and
\begin{align*}g^{'}(x)&=\sum\limits_{k=0}^{n}\binom{n}{k}(-1)^kk!(n-k)x^{n-k-1}\psi_{p,q}^{(n-k-1)}(x+1)+\notag \\
\notag &+\sum\limits_{k=0}^{n}\binom{n}{k}(-1)^kk!x^{n-k}\psi_{p,q}^{(n-k)}(x+1)\notag\\
\notag &=\sum\limits_{k=0}^{n-1}\binom{n}{k}(-1)^kk!(n-k)x^{n-k-1}\psi_{p,q}^{(n-k-1)}(x+1)+\notag \\
\notag &+x^n\psi_{p,q}^{(n)}(x+1)+\sum\limits_{k=0}^{n}\binom{n}{k}(-1)^kk!x^{n-k}\psi_{p,q}^{(n-k)}(x+1)\notag\\
\notag &=\sum\limits_{k=0}^{n-1}\binom{n}{k}(-1)^kk!(n-k)x^{n-k-1}\psi_{p,q}^{(n-k-1)}(x+1)+\notag \\
\notag &+x^n\psi_{p,q}^{(n)}(x+1)+\sum\limits_{k=0}^{n-1}\binom{n}{k+1}(-1)^{k+1}(k+1)!x^{n-k-1}\psi_{p,q}^{(n-k-1)}(x+1)\notag\\
\notag &=\sum\limits_{k=0}^{n-1}\Big[\binom{n}{k}(n-k)-\binom{n}{k+1}(k+1)\Big](-1)^kk!x^{n-k-1}\psi_{p,q}^{(n-k-1)}(x+1)\notag\\
\notag &+x^n\psi_p^{(n)}(x+1)=x^n\psi_{p,q}^{(n)}(x+1).
\end{align*}
If $n$ is odd, then for $x>0,$
$$g^{'}(x)>0\Rightarrow g(x)>g(0)=0\Rightarrow (\ln f(x))^{(n)}<0\Rightarrow$$
$$\Rightarrow (-1)^n(\ln f(x))^{(n)}>0.$$
If $n$ is even, then for $x>0,$
$$g^{'}(x)<0\Rightarrow g(x)<g(0)=0\Rightarrow (\ln f(x))^{(n)}>0\Rightarrow$$
$$\Rightarrow (-1)^n(\ln f(x))^{(n)}>0.$$
Hence,
$$(-1)^n(\ln f(x))^{(n)}>0$$
for all real $x\in (0,\infty)$ and all integers $n\geq 1.$ The proof is completed.
\end{proof}

\begin{remark}Let $p$ tend to $\infty,$ then we obtain Theorem 1 of \cite{chen}. Let $q$ tend to 1, then we obtain Theorem 2.1 of \cite{faton}.
\end{remark}

Let $s$ and $t$ be two real numbers with $s\not= t, \alpha =\min \{s,t\}$ and $\beta \geq -\alpha,$ for $x\in (-\alpha, \alpha),$ define
$$h_{\beta,p,q}(x)=\left\{\begin{array}{cc}
\Big[\frac{\Gamma _{p,q}(\beta+t)}{\Gamma_{p,q}(\beta+s)}\cdot\frac{\Gamma _{p,q}(x+s)}{\Gamma_{p,q}(x+t)}\Big]^{\frac{1}{x-\beta}}&x\not= \beta\\[10pt]
\exp [\psi_{p,q} (\beta+s)-\psi_{p,q}(\beta+t)]&x=\beta
\end{array}\right.$$
The following theorem is a generalization of a result of \cite{Q}.

\begin{theorem} The function $h_{\beta,p,q}(x)$ is logarithmically completely monotonic on\break $(-\alpha, +\infty)$ if $s>t.$
\end{theorem}
\begin{proof}
For $x\not= \beta,$ taking  logarithm of the function $h_{\beta,p}(x)$ gives
\begin{align}
\ln h_{\beta,p,q}(x)&=\frac{1}{x-\beta}\Big[ \ln\frac{\Gamma _{p,q}(\beta+t)}{\Gamma_{p,q}(\beta+s)}-\ln\frac{\Gamma _{p,q}(x+s)}{\Gamma_{p,q}(x+t)}\Big]\notag\\
\notag &=\frac{\ln\Gamma _{p,q}(x+s)-\ln\Gamma_{p,q}(\beta+s)}{x-\beta}-\frac{\ln\Gamma _{p,q}(x+t)-\ln\Gamma_{p,q}(\beta+t)}{x-\beta}\notag \\
\notag &=\frac{1}{x-\beta}\int\limits_{\beta}^{x}\psi_{p,q}(u+s)du-\frac{1}{x-\beta}\int\limits_{\beta}^{x}\psi_{p,q}(u+t)du\notag \\
\notag &=\frac{1}{x-\beta}\int\limits_{\beta}^{x}[\psi_{p,q}(u+s)-\psi_{p,q}(u+t)]du\notag\\
\notag &=\frac{1}{x-\beta}\int\limits_{\beta}^{x}\int\limits_{t}^{s}\psi^{'}_{p,q}(u+v)dvdu\notag\\
\notag &=\frac{1}{x-\beta}\int\limits_{\beta}^{x}\varphi_{p,q,s,t}(u)du\notag\\
\notag &=\int\limits_{0}^{1}\varphi_{p,q,s,t}((x-\beta)u+\beta)du,
\end{align}
and by differentiating $\ln h_{\beta,p,q}(x)$ with respect to $x,$
\begin{equation}\label{ekj}[\ln h_{\beta,p,q}(x)]^{(k)}=\int\limits_{0}^{1}u^k\varphi^{(k)}_{p,q,s,t}((x-\beta)u+\beta)du.\end{equation}
If $x=\beta$ formula (\ref{ekj}) is valid.
Since  functions $\psi^{'}_{p,q}$ and $\varphi_{p,q,s,t}$ are  completely monotonic in $(0,\infty)$ and $(-t,\infty)$ respectively, then $(-1)^i[\varphi_{p,q,s,t}(x)]^{(i)}\geq 0$  holds for $n\in (-t,\infty)$ for any nonnegative integer $i.$
Thus $$(-1)^k[\ln h_{\beta,p,q}(x)]^{(k)}=\int\limits_{0}^{1}u^k(-1)^k\varphi^{(k)} _{p,q,s,t}((x-\beta)u+\beta)du\geq 0$$ in $(-t,\infty)$ for $k\in N.$ The proof is completed.
\end{proof}

\section{Application of $\Gamma_{p,q}(x)$ function}
In the following, we give the $\Gamma_{p,q}$ analogue of results from \cite{SHA}.
Since the proofs are almost similar, we omit them.

\begin{lemma}\label{lem22} Let $a,b,c,d,e$ be real numbers such that $a+bx>0,\ d+ex>0$ and
$a+bx \leq d+ex$. Then
\begin{equation}\label{eq2.2}
\psi_{p,q}(a+bx)-\psi_{p,q}(d+ex)\leq 0.
\end{equation}
\end{lemma}

\begin{lemma}\label{lem23} Let $a,b,c,d,e,f$ be real numbers such that $a+bx>0, d+ex>0,
a+bx\leq d+ex$ and $ef\geq bc>0$. If

(i) $\psi_{p,q}(a+bx)>0$ or

(ii) $\psi_{p,q}(d+ex)>0$

then
\begin{equation}\label{eq2.3}
bc\psi_{p,q}(a+bx)- ef\psi_{p,q}(d+ex)\leq 0.
\end{equation}
\end{lemma}

\begin{lemma}\label{lem24} Let $a,b,c,d,e,f$ be real numbers such that $a+bx>0, d+ex>0,
a+bx\leq d+ex$ and $bc\geq ef>0$. If

(i) $\psi_{p,q}(d+ex)<0$ or

(ii) $\psi_{p,q}(a+bx)<0$

then
\begin{equation}\label{eq2.4}
bc\psi_{p,q}(a+bx)- ef\psi_{p,q}(d+ex)\leq 0.
\end{equation}
\end{lemma}

\begin{theorem}\label{theorem2} Let $f_1$ be a function defined by
\begin{equation}\label{eq2.5}
f_1(x)=\frac{\Gamma_{p,q}(a+bx)^{c}}{\Gamma_{p,q} (d+ex)^{f}},\quad x\geq 0
\end{equation}
where \ $a,b,c,d,e,f$ are real numbers such that: $a+bx>0, d+ex>0,
a+bx\leq d+ex, ef\geq bc>0$. If $\psi_{p,q}(a+bx)>0$ or $\psi_{p,q}(d+ex)>0$
then the function $f_1$ is decreasing for $x \geq 0$ and for $x\in
[0,1]$ the following double inequality holds:
\begin{equation}\label{eq2.6}
\frac{\Gamma_{p,q}(a+b)^{c}}{\Gamma_{p,q}(d+e)^{f}} \leq
\frac{\Gamma_{p,q}(a+bx)^{c}}{\Gamma_{p,q} (d+ex)^{f}} \leq \frac{\Gamma_{p,q}
(a)^{c}}{\Gamma_{p,q} (d)^{f}}.
\end{equation}
\end{theorem}

In a similar way, using Lemma \ref{lem24}, it is easy to prove the
following Theorem.

\begin{theorem}\label{theorem3}
Let $f_1$ be a function defined by
\begin{equation}\label{eq2.7}
f_1(x)=\frac{\Gamma_{p,q}(a+bx)^{c}}{\Gamma_{p,q}(d+ex)^{f}}, \quad x\geq 0,
\end{equation}
where \ $a,b,c,d,e,f$ are real numbers such that: $a+bx>0, d+ex>0,
a+bx\leq d+ex, bc\geq ef>0$. If $\psi_{p,q}(d+ex)<0$ or $\psi_{p,q}(a+bx)<0$
then the function $f_1$ is decreasing for $x\geq 0$ and for $x\in
[0,1]$ the inequality \eqref{eq2.6} holds.
\end{theorem}

{\bf Acknowledgements.}\, We would like to thank  Feng Qi and Armend Shabani for several corrections and suggestions.


\begin{thebibliography}{99}
\bibitem{AL0}
H. Alzer, {\em On some inequalities for the gamma and psi function},{\it Math. Comp. }{ 66} (1997) 373-389.

\bibitem{AP} T.M. Apostol, Introduction to Analytic Number Theory, Springer, 1976.

\bibitem{AS} M. Abramowitz and I.A. Stegun, Handbook of Mathematical Functions
with Formulas and Mathematical Tables, Dover, NewYork, 1965.

\bibitem{As}
R. Askey, {\em The $q$-gamma and $q$-beta functions},{\it Applicable Anal.} {8}(2) (1978/79) 125--141.

\bibitem{BO}
S. Bocher, {\it Harmonic Analysis and the theory of Probability}, Dover Books, 2005.

\bibitem{chen} Ch.-P. Chen and F. Qi, Logarithmically completely monotonic functions relating to the
gamma functions,{\it J. Math. Anal. Appl.} 321 (2006), 405-411.

\bibitem{Integrali} M.E.H. Ismail and M.E. Muldoon, Inequalites and monotonicity properties for gamma and q-gamma functions, {\it ISNM Aproximation and computation} 119 (1994), 309-323.

\bibitem{K}T. Kim, {\em On a $q$-analogue of the $p$-adic log gamma functions and
related integrals}, {\it J. Number Theory} { 76} (1999) 320-329.

\bibitem{K2}
T. Kim, {\em A note on the $q$-multiple zeta functions}, {\it Advan. Stud.
Contemp. Math. }{ 8} (2004) 111-113.

\bibitem{KR}
T. Kim and S.H. Rim, {\em A note on the $q$-integral and $q$-series},
{\it Advanced Stud. Contemp. Math. }{ 2} (2000) 37--45.

\bibitem{Valmir} V. Krasniqi and S. Guo, Logarithmically Completely monotonic functions involving generalized gamma and $q$-gamma functions,{\it J. Inequal. Spec. Funct.}, 1 (2010), 8 - 16.

\bibitem{Valmir1} V. Krasniqi and A. Shabani, Convexity properties and inequalities for a generalized gamma functions, {\it Appl. Math. E-Notes}, 10(2010), 27-35.

\bibitem{Valmir2}
V. Krasniqi, T. Mansour and A.Sh. Shabani, Some Monotonicity Properties and Inequalities for the Gamma and Rimann Zeta Functions, {\it Math. Commun,} Vol. 15, No. 2, pp. 365-376 (2010).

\bibitem{faton} V. Krasniqi and F. Merovci, Logarithmically completely monotonic functions involving the Generalized Gamma Function, {\it Le Matematiche }
Vol. LXV (2010)  Fasc. II, pp. 15-23.

\bibitem{Q}F. Qi, Three class of logarithmically completely monotonic functions involving the gamma
and psi functions,{\it  Integral Transform Spec . Funct. }18 (2007), 503-509 .

\bibitem{IM}
M.E.H. Ismail and M.E. Muldoon, {\em Inequalities and monotonicity
properties for gamma and $q$-gamma functions}, in: R.V.M. Zahar
(Ed.), {\it Approximation and Computation, International Series of
Numerical Mathematics,} vol. 119, Birkh\"auser, Boston, MA, 1994,
309--323.

\bibitem{MI} G.H. Hardy, J.E. Littlewood and G. P\'{o}lya, Inequalities, Cambridge University Press, 1988.

\bibitem{QI}
F. Qi and Ch.-P. Chen, {\em A complete monotonicity property of the Gamma function}, {\it J. Math. Anal. Appl.} { 296}
(2004), 603--607.

\bibitem{SA} J. Sandor, Selected Chapters of Geometry, Analysis and Number Theory, {\it RGMIA Monographs, Victoria University,} 2005.

\bibitem{SKS}
H.M. Srivastava, T. Kim and Y. Simsek, {\em $q$-Bernoulli numbers and
polynomials associated with multiple $q$-zeta functions and basic
L-series},{\it Russian J. Math. Phys. }{12} (2005) 241--268.

\bibitem{SHA} A.Sh. Shabani, Generalization of some inequalities for the Gamma function, {\it Mathematical Communications, }13(2008),271--275.
\end{thebibliography}
\end{document}